\documentclass{article}
\usepackage[english]{babel}
\usepackage{amsmath}
\usepackage{amssymb}
\usepackage{amsfonts}
\usepackage{amsthm}
\usepackage{graphicx}

\usepackage{url}

\def\thtext#1{
  \catcode`@=11
  \gdef\@thmcountersep{. #1}
  \catcode`@=12
}

\def\threst{
  \catcode`@=11
  \gdef\@thmcountersep{.}
  \catcode`@=12
}

\theoremstyle{plain}
\newtheorem{thm}{Theorem}[section]
\newtheorem{prop}{Proposition}[section]
\newtheorem{cor}[prop]{Corollary}
\newtheorem{lem}[prop]{Lemma}

\theoremstyle{definition}
\newtheorem{rk}[prop]{Remark}
\newtheorem{dfn}[prop]{Definition}
\newtheorem{constr}[prop]{Construction}

\catcode`@=11
\def\.{.\spacefactor\@m}
\catcode`@=12

\def\a{\alpha}
\def\dl{\delta}
\def\D{\Delta}
\def\g{\gamma}
\def\l{\lambda}
\def\s{\sigma}

\def\:{\colon}
\def\rom#1{\emph{#1}}
\def\({\rom(}
\def\){\rom)}
\def\sm{\setminus}
\def\ss{\subset}

\def\x{\times}

\def\diam{\operatorname{diam}}
\def\dis{\operatorname{dis}}
\def\mf{\operatorname{mf}}
\def\mst{\operatorname{mst}}
\def\MST{\operatorname{MST}}
\def\opt{{\operatorname{opt}}}
\def\smt{\operatorname{smt}}
\def\SMT{\operatorname{SMT}}

\def\cD{{\cal D}}
\def\cM{{\cal M}}
\def\cP{{\cal P}}
\def\cR{{\cal R}}
\def\cT{{\cal T}}

\begin{document}
\title{Calculation of Minimum Spanning Tree Edges Lengths using Gromov--Hausdorff Distance}
\author{Alexey~A.~Tuzhilin}
\maketitle

\begin{abstract}
In the present paper we show how one can calculate the lengths of edges of a minimum spanning tree constructed for a finite metric space, in terms of the Gromov--Hausdorff distances from this space to simplices of sufficiently large diameter. Here by simplices we mean finite metric spaces all of whose nonzero distances are the same. As an application, we reduce the problems of finding a Steiner minimal tree length or a minimal filling length to maximization of the total distance to some finite number of simplices considered as points of the Gromov--Hausdorff space.
\end{abstract}

\section*{Introduction}
\markright{\thesection.~Introduction}

In this paper we display a marvelous relation between the edges lengths of a minimum spanning tree constructed for a finite metric space, and the Gromov--Hausdorff distances from this space to simplices of sufficiently large diameter, where by simplices we mean finite metric spaces with the same distances between their distinct points. This relation follows from natural minimax formula for calculating the lengths of edges of a minimum spanning tree, which is an analogue of that is usually used for calculating the eigenvalues of a self-adjoint operator. As a corollary, we get some formulaes for calculating the lengths of minimum spanning trees, Steiner minimal trees, and minimal fillings, by means of Gromov--Hausdorff distances.

The author expresses his deep gratitude to his colleague and permanent co-author Alexander O.~Ivanov for useful remarks.

\section{Preliminaries}
\markright{\thesection.~Preliminaries}
Let $X$ be an arbitrary set. Denote by $\#X$ the \emph{cardinality\/} of the set $X$. We are mainly interested in cardinalities of finite sets, thus, for such $X$, the $\#X$ is just the number of elements in $X$.

Let $X$ be an arbitrary metric space. The distance between its points $x$ and $y$ is denoted by $|xy|$. If $A,B\ss X$ are nonempty subsets, then we put $|AB|=\inf\bigl\{|ab|:a\in A,\,b\in B\bigr\}$. If $A=\{a\}$, then we simply write $|aB|=|Ba|$ instead of $|\{a\}B|=|B\{a\}|$.

For every point $x\in X$ and a number $r>0$, we denote by $U_r(x)$ the open ball of center $x$ and radius $r$; for any nonempty $A\ss X$ and $r>0$ we put $U_r(A)=\cup_{a\in A}U_r(a)$.

For nonempty $A,\,B\ss X$ we put
$$
d_H(A,B)=\inf\bigl\{r>0:A\ss U_r(B)\ \&\ B\ss U_r(A)\bigr\}=\max\{\sup_{a\in A}|aB|,\sup_{b\in B}|Ab|\}.
$$
This value if called the \emph{Hausdorff distance between $A$ and $B$}. It is well-known~\cite{BurBurIva} that the Hausdorff distance restricted to the family of all nonempty closed bounded subsets of $X$ is a metric.

Let $X$ and $Y$ be metric spaces. A triple $(X',Y',Z)$ that consists of a metric space $Z$ and its subsets $X'$ and $Y'$ isometric to $X$ and $Y$, respectively, is called a \emph{realization of the pair $(X,Y)$}. The \emph{Gromov--Hausdorff distance $d_{GH}(X,Y)$ between $X$ and $Y$} is the least lower bound of real numbers $r$ such that there exists a realization $(X',Y',Z)$ of the pair $(X,Y)$ with $d_H(X',Y')\le r$. It is well-known~\cite{BurBurIva} that the $d_{GH}$ restricted to the family $\cM$ of isometry classes of compact metric spaces is a metric.

For various calculations of the Gromov--Hausdorff distances, it is convenient to use the technique of correspondences.

Let $X$ and $Y$ be arbitrary nonempty sets. Recall that a \emph{relation\/} between the sets $X$ and $Y$ is a subset of the Cartesian product $X\x Y$. The set of all \textbf{nonempty\/} relations between $X$ and $Y$ we denote by $\cP(X,Y)$. Let us look at each relation $\s\in\cP(X,Y)$ as a multivalued mapping whose domain may be less than the whole $X$. Then, similarly with the case of mappings, for any $x\in X$ and any $A\ss X$ there are defined their images $\s(x)$ and $\s(A)$, and for any $y\in Y$ and any $B\ss Y$ their preimages $\s^{-1}(y)$ and $\s^{-1}(B)$, respectively.

A relation $R\in\cP(X,Y)$ is called a \emph{correspondence\/} if the restrictions of the canonical projections $\pi_X\:(x,y)\mapsto x$ and $\pi_Y\:(x,y)\mapsto y$ onto $R$ are surjective. The set of all correspondences between $X$ and $Y$ we denote by $\cR(X,Y)$.

Let $X$ and $Y$ be arbitrary metric spaces. The \emph{distortion $\dis\s$ of a relation $\s\in\cP(X,Y)$} is the value
$$
\dis\s=\sup\Bigl\{\bigl||xx'|-|yy'|\bigr|: (x,y),(x',y')\in\s\Bigr\}.
$$

\begin{prop}[\cite{BurBurIva}]
For any metric spaces $X$ and $Y$ it holds
$$
d_{GH}(X,Y)=\frac12\inf\bigl\{\dis R:R\in\cR(X,Y)\bigr\}.
$$
\end{prop}

If $X$ and $Y$ are finite metric spaces, then the set $\cR(X,Y)$ is finite, thus, there exists an $R\in\cR(X,Y)$ such that $d_{GH}(X,Y)=\frac12\dis R$. Every such $R$ is called \emph{optimal}. Notice that optimal correspondences exist also for arbitrary compact metric spaces $X$ and $Y$, see~\cite{IvaIliadisTuz}. The set of all optimal correspondences between $X$ and $Y$ we denote by $\cR_\opt(X,Y)$.

The next well-know fact may be easily proved by means of the technique of correspondences. For any metric space $X$ and a real number $\l>0$, we denote by $\l X$ the metric space obtained from $X$ by means of $\l$-times extension of all the distances in $X$.

\begin{prop}[\cite{BurBurIva}]\label{prop:scale}
For any $X,Y\in\cM$ and any $\l>0$ we have
$$
d_{GH}(\l X,\l Y)=\l\,d_{GH}(X,Y).
$$
Moreover, for $\l\ne1$ the only space which is not changed under such modification is the one-point space. In other words, the operation of multiplication of metric on a real number $\l>0$ is a homothety of the space $\cM$ with center at the one-point space.
\end{prop}

Let $G=(V,E)$ be an arbitrary graph with the vertices set $V$ and the edges set $E$. We say that the graph $G$ is \emph{defined on a metric space $X$}, if $V\ss X$. For every such a graph, there are defined the \emph{lengths $|e|$} of its \emph{edges $e=vw$} as the distances $|vw|$ between the ends $v$ and $w$ of these edges; the \emph{length $|G|$} of the \emph{graph $G$} itself is the sum of its edges lengths.

If $M\ss X$ is an arbitrary nonempty finite subset, and $G=(V,E)$ is a connected graph on $X$, then we say that $G$ \emph{joins $M$} if $M\ss V$. The least lower bound of the lengths of the connected graphs that join $M\ss X$ is called the \emph{length of Steiner minimal tree on $M$} and is denoted by $\smt(M,X)$. Every connected graph $G$ that joins $M$ and such that $|G|=\smt(M,X)$ is a tree which is called a \emph{Steiner minimal tree on $M$}. The set of all Steiner minimal trees on $M$ we denote by $\SMT(M,X)$. Notice that the set $\SMT(M,X)$ may be empty, and that $\SMT(M,X)$ and $\smt(M,X)$ depend not only on the distances between the points of $M$, but also on the geometry of the ambient space $X$: isometrical $M$ that belong to different metric spaces $X$, may be joined by Steiner minimal trees of nonequal lengths. Some details on the theory of Steiner minimal trees one can find, for example, in~\cite{IvaTuz} or~\cite{HwRiW}.

Denote by $\cT(M,X)$ the set of all trees $G=(V,E)$ on $X$ that join $M$ and satisfy $\#V\le2\#M-2$. The next result follows easily from~\cite{IvaTuz}.

\begin{prop}[\cite{IvaTuz}]\label{prop:Steiner-regular}
For any metric space $X$ and any its nonempty finite subset $M$ we have
$$
\smt(M,X)=\inf_{G\in\cT(M,T)}|G|.
$$
\end{prop}

Let $M$ be a finite metric space. Define a value $\mst(M)$ as the length of the shortest tree among the trees of the form $(M,E)$. This value is called the \emph{length of minimum spanning tree on $M$}; a tree $G=(M,E)$ such that $|G|=\mst(M)$ is called a \emph{minimum spanning tree on $M$}. Notice that for any $M$ there exists a minimum spanning tree on $M$. The set of all minimum spanning trees on $M$ we denote by $\MST(M)$.

Proposition~\ref{prop:Steiner-regular} may be naturally rewritten in terms of minimum spanning trees. For any finite metric space $M$ we denote by $\cM(M)\ss\cM$ the set of isometry classes of all metric spaces $V$ such that $M\ss V$ and $\#V\le2\#M-2$. If $M$ is a finite subset of a metric space $X$, then we denote by $\cM(M,X)$ the subset of $\cM(M)$ that consists of isometry classes of all subsets $V\ss X$ such that $M\ss V$ and $\#V\le2\#M-2$.

\begin{prop}[\cite{IvaTuz}]\label{prop:smt-in-terms-mst}
For any metric space $X$ and any its nonempty finite subset $M$ we have
$$
\smt(M,X)=\inf_{V\in\cM(M,X)}\mst(V).
$$
\end{prop}

Now, let us fix a finite metric space $M$. Consider isometrical embeddings of $M$ into various metric spaces $X$, and for every such embedding calculate the length of Steiner minimal tree on the image of $M$. The ``shortest'' length of these Steiner minimal trees we call the \emph{length of minimal filling for $M$}, and this value we denote by $\mf(M)$. To overcome the problems like Cantor's paradox, let us give a more accurate definition. We define the value $\mf(M)$ to be the greatest lower bound of the real numbers $r$ such that there exist a metric space $X$ and an isometric embedding $\mu\:M\to X$ with $\smt\bigl(\mu(M),X\bigr)\le r$.

The next result easily follows from~\cite{ITMinFil}.

\begin{prop}[\cite{ITMinFil}]\label{prop:mf-in-terms-mst}
For any finite metric space $M$ it holds
$$
\mf(M)=\inf_{V\in\cM(M)}\mst(V).
$$
\end{prop}

\section{$\mst$-spectrum of finite metric space}
\markright{\thesection.~$\mst$-spectrum of finite metric space}
Notice that minimum spanning tree may be defined not uniquely. For $G\in\MST(M)$ we denote by $\s(G)$ the vector whose coordinates are the lengths of edges of the tree $G$, ordered descending. The next result is well-known, however, we present its proof for the reason of completeness.

\begin{prop}\label{prop:mst-spect}
For any $G_1,G_2\in\MST(M)$ we have $\s(G_1)=\s(G_2)$.
\end{prop}

\begin{proof}
Recall a standard algorithm of transferring from one minimum spanning tree to another one, see~\cite{Harary}.

Let $G_1\ne G_2$. Put $G_i=(M,E_i)$, then $E_1\ne E_2$ and $\#E_1=\#E_2$, therefore, there exists $e\in E_2\sm E_1$. The graph $G_1\cup e$ has a cycle $C$ that contains the edge $e$. The remaining edges of $C$ are not longer than $e$, otherwise $G_1\not\in\MST(M)$. The forest $G_2\sm e$ consists of two trees whose vertices sets we denote by $V'$ and $V''$. Clearly that $M=V'\sqcup V''$. The cycle $C$ contains an edge $e'\ne e$ that joins a vertex from $V'$ with a vertex from $V''$. This edge does not belong to $E_2$, otherwise $G_2$ contains a cycle. Hence, $e'\in E_1\sm E_2$.

The graph $G_2\cup e'$ contains a cycle $C'$. By the choice of $e'$, the cycle $C'$ contains the edge $e$ as well. Similarly with that was discussed above, the length of $e$ is less than or equal to the length of $e'$, otherwise $G_2\not\in\MST(M)$. Therefore, $|e|=|e'|$.

Let us modify the graph $G_1$ by changing the edge $e'$ to $e$. The result is a tree $G_1'$ of the same length, i.e., $G_1$ is a minimum spanning tree as well. Notice that $G_1'$ and $G_2$ have one common edge more than the trees $G_1$ and $G_2$. Thus, we can transfer $G_1$ to $G_2$ while finite number of steps, passing through minimum spanning trees. It remains to notice that $\s(G_1')=\s(G_1)$, hence, $\s(G_1)=\s(G_2)$.
\end{proof}

Proposition~\ref{prop:mst-spect} explains correctness of the following definition.

\begin{dfn}
For any finite metric space $M$, we denote by $\s(M)$ the vector $\s(G)$ for arbitrary $G\in\MST(M)$, and we call this vector by the \emph{$\mst$-spectrum of the space $M$}.
\end{dfn}

\begin{constr}
For any set $M$ we denote by $\cD_k(M)$ the family of all possible partitions of the $M$ to $k$ its nonempty subsets. Suppose now that $M$ is a metric space and $D=\{M_1,\ldots,M_k\}\in\cD_k(M)$. Put
$$
\a(D)=\min\bigl\{|M_iM_j|:i\ne j\bigr\}.
$$
\end{constr}

\begin{thm}\label{thm:spect-calc}
Let $M$ be a finite metric space and $\s(M)=(\s_1,\ldots,\s_{n-1})$. Then
$$
\s_k=\max\bigl\{\a(D):D\in\cD_{k+1}(M)\bigr\}.
$$
\end{thm}

\begin{proof}
Let $G=(M,E)\in\MST(M)$ and the set $E$ be ordered in such a way that $|e_i|=\s_i$. Denote by $D=\{M_1,\ldots,M_{k+1}\}$ the partition of $M$ to the vertices sets of trees of the forest $G\sm\cup_{i=1}^ke_i$.

\begin{lem}\label{lem:canon-part}
It holds $\a(D)=|e_k|$.
\end{lem}

\begin{proof}
Indeed, let us choose arbitrary $M_i$ and $M_j$, $i\ne j$, some their points $P_i$ and $P_j$, respectively, and denote by $\g$ the unique path in $G$ that joins $P_i$ and $P_j$. Then $\g$ contains an edge $e_p$, $1\le p\le k$. However, by minimality of the tree $G$, we have $|P_iP_j|\ge|e_p|\ge\min_{i=1}^k|e_i|=|e_k|$, hence, $|M_iM_j|\ge|e_k|$, so $\a(D)\ge|e_k|$. On the other hand, if $i$ and $j$ are chosen in such a way that $e_k$ joins $M_i$ and $M_j$, then we get $\a(D)\le|M_iM_j|=|e_k|$.
\end{proof}

Now, consider an arbitrary partition $D'=\{M_1',\ldots,M_{k+1}'\}$.

\begin{lem}\label{lem:noncanon-part-less}
We have $\a(D')\le\a(D)$.
\end{lem}

\begin{proof}
By Lemma~\ref{lem:canon-part}, it suffices to show that $\a(D')\le|e_k|$. Denote by $E'$ the set consisting of all edges $e_p\in E$ such that there exist $M_i'$ and $M_j'$, $i\ne j$, and $e_p$ joins $M_i'$ and $M_j'$. Since $G$ is connected, the set $E'$ consists at least of $k$ edges, because otherwise the indices set $\{1,\ldots,k+1\}$ can be partitioned to two nonempty subsets $I$ and $J$ such that the sets $\cup_{i\in I}M_i'$ and $\cup_{j\in J}M_j'$ that generate a partition of $M$, can not be joined by an edge from $E$. On the other hand, if some $M_i'$ and $M_j'$ are joined by an edge $e'\in E$, then $|M_i'M_j'|\le|e'|$, hence,  $\a(D')=\min|M_i'M_j'|\le\min_{e'\in E'}|e'|\le|e_k|$.
\end{proof}

Lemma~\ref{lem:noncanon-part-less} completes the proof of the theorem.
\end{proof}

\section{Calculation of $\mst$-spectrum using Gromov--Hausdorff distances}
\markright{\thesection.~Calculation of $\mst$-spectrum using Gromov--Hausdorff distances}

\begin{constr}
For $\l>0$ we denote by $\l\D_n$ the metric space that consists of $\{1,\ldots,n\}$ and all its nonzero distances equal $\l$. If $\l=1$, then the space $\l\D_n$ we simply denote by $\D_n$.
\end{constr}

In the present section we show that the $\mst$-spectrum of an $n$-points metric space can be represented as a linear function on the Gromov--Hausdorff distances from this space to the simplices $\l\D_2,\ldots,\l\D_{n+1}$ for $\l\ge2\diam X$.

\begin{thm}\label{thm:spectrum-as-GH}
Let $X\in\cM$ be a finite metric space, $\s(X)=(\s_1,\ldots,\s_{n-1})$, $\diam X\le1/2$. Then $\s_k=1-2d_{GH}(X,\D_{k+1})$.
\end{thm}

\begin{proof}
Let $R\in\cR(\D_{k+1},X)$, and put $X_i=R(i)$. We denote by $\dl_R$ the number that equals $1$ if there exists an $x\in X$ such that $\#R^{-1}(x)\ge2$, and $0$ otherwise. Notice that $\dl_R=0$ iff $\{X_i\}$ is a partition of $X$.

By definition of the distortion, we have
\begin{multline*}
\dis R=\max
\begin{cases}
\dl_R,\\
\max\{\diam X_i:i=1,\ldots,k+1\},\\
\max\bigl\{1-|x_ix_j|:x_i\in X_i,\,x_j\in X_j,\,i\ne j\bigr\}
\end{cases}
\\
=\max\Bigl[\dl_R,\max\{\diam X_i:i=1,\ldots,k+1\},1-\min\bigl\{|X_iX_j|:i\ne j\bigr\}\Bigr].
\end{multline*}
Since, by assumptions, $\diam X_i\le1/2$ for all $i$, and $1-\min\bigl\{|X_iX_j|:i\ne j\bigr\}<1$, then for every $R$ such that $\{X_i\}$ is a partition of $X$, we have $\dis R<1$. Since for the cases under consideration it holds $k+1\le n$, then those $R$ for which $\{X_i\}$ form partitions of $X$, do exist, hence, for each $R\in\cR_\opt(\D_{k+1},X)$ we have $\dl_R=0$.

Denote by $\cR_p(\D_{k+1},X)$ the set of those correspondences from $\cR(\D_{k+1},X)$ such that $\{X_i\}$ is a partition of $X$. From the above discussion, it follows that
$$
\cR_\opt(\D_{k+1},X)\ss\cR_p(\D_{k+1},X),
$$
hence,
$$
2d_{GH}(\D_{k+1},X)=\min\bigl\{\dis R:R\in\cR_p(\D_{k+1},X)\bigr\}.
$$
Moreover, for every partition $D=\{X_1,\ldots,X_{k+1}\}\in\cD_{k+1}(X)$ the set $\cup_{i=1}^{k+1}\{i\}\x X_i$ is a correspondence from $\cR_p(\D_{k+1},X)$. In other words, all partitions $D\in\cD_{k+1}(X)$ are generated by correspondences from $\cR_p(\D_{k+1},X)$.

Let us choose an arbitrary $R\in\cR_p(\D_{k+1},X)$. Since $|X_iX_j|\le\diam X\le1/2$, then
$$
1-\min\bigl\{|X_iX_j|:i\ne j\bigr\}\ge1/2\ge\max\{\diam X_i:i=1,\ldots,k+1\},
$$
thus,
$$
\dis R=1-\min\bigl\{|X_iX_j|:i\ne j\bigr\}=1-\a\bigl(\{X_i\}\bigr).
$$
By Theorem~\ref{thm:spect-calc}, we get
\begin{multline*}
2d_{GH}(\D_{k+1},X)=\min\Bigl[1-\a\bigl(\{X_i\}\bigr):R\in\cR_p(\D_{k+1},X)\Bigr]=\\ =1-\max\bigl[\a(D):D\in\cD_{k+1}(X)\bigr]=1-\s_k.
\end{multline*}
\end{proof}

The next result follows immediately from Theorem~\ref{thm:spectrum-as-GH} and Proposition~\ref{prop:scale}.

\begin{cor}\label{cor:spectrum-as-GH}
Let $X\in\cM$ be a finite metric space, $\s(X)=(\s_1,\ldots,\s_{n-1})$, $\l\ge2\diam X$. Then $\s_k=\l-2d_{GH}(X,\l\D_{k+1})$.
\end{cor}

\begin{rk}\label{rk:pseudo}
It is easy to see that, under assumptions of Corollary~\ref{cor:spectrum-as-GH}, for each $m>n$ it holds $d_{GH}(X,\l\D_m)=\l/2$, therefore, the right hand side of the equality from this corollary vanishes.
\end{rk}

From Remark~\ref{rk:pseudo} and Theorem~\ref{thm:spectrum-as-GH} we immediately get a formula for calculating the length of a minimum spanning tree.

\begin{cor}\label{cor:spectrum-as-GH-mod}
Let $X\in\cM$ be a finite metric space and $\l\ge2\diam X$, then
$$
\mst X=\sum_{k=1}^\infty\bigl[\l-2d_{GH}(X,\l\D_{k+1})\bigr]=\l(\#X-1)-2\sum_{k=1}^{\#X-1}d_{GH}(X,\l\D_{k+1}).
$$
\end{cor}

\section{The length of Steiner minimal tree}
\markright{\thesection.~The length of Steiner minimal tree}
Let $X$ be an arbitrary metric space, and $M\ss X$ be its $n$-points subset. Recall that by $\cM(M,X)\ss\cM$ we denote the family of isometry classes of subsets $V\ss X$ such that $M\ss V$ and $\#V\le2n-2$. Let us choose a real number $d$ and put $\cM(M,X,d)=\{V\in\cM(M,X):\diam V\le d\}$.

Corollary~\ref{cor:spectrum-as-GH-mod} and Proposition~\ref{prop:smt-in-terms-mst} imply the following result.

\begin{cor}
Let $M$ be a finite subset of a metric space $X$, and $n=\#M$. Choose an arbitrary $d>\smt(M)$, e.g., $d>(n-1)\diam M$, and an arbitrary $\l\ge2d$, then
$$
\smt(M,X)=\inf\biggl\{\sum_{k=1}^\infty\bigl[\l-2d_{GH}(V,\l\D_{k+1})\bigr]:V\in\cM(M,X,d)\biggr\}.
$$
\end{cor}

\section{The length of minimal filling}
\markright{\thesection.~The length of minimal filling}
Let $M$ be an arbitrary $n$-points metric space. Recall that by $\cM(M)\ss\cM$ we denote the family of isometry classes of metric spaces $V$ such that $M\ss V$ and $\#V\le2n-2$. Choose an arbitrary real number $d$ and put $\cM(M,d)=\{V\in\cM(M):\diam V\le d\}$.

Corollary~\ref{cor:spectrum-as-GH-mod} and Proposition~\ref{prop:mf-in-terms-mst} imply the following result.

\begin{cor}
Let $M$ be a finite metric space, and $n=\#M$. Choose an arbitrary $d>\mf(M)$, e.g., $d>(n-1)\diam M$, and an arbitrary $\l\ge2d$, then
$$
\mf(M)=\inf\biggl\{\sum_{k=1}^\infty\bigl[\l-2d_{GH}(V,\l\D_{k+1})\bigr]:V\in\cM(M,d)\biggr\}.
$$
\end{cor}

\end{document}